\documentclass[12pt,a4paper]{amsart}
\usepackage[utf8]{inputenc}
\usepackage[lmargin=2.5cm,tmargin=2cm,bmargin=2cm,rmargin=2.5cm]{geometry}
\usepackage{amsfonts}
\usepackage{amssymb}
\usepackage{fix-cm}
\usepackage{amscd,amssymb,stmaryrd}
\usepackage{amsmath}
\usepackage{graphicx}
\usepackage{subfigure}
\usepackage{fancybox}
\usepackage{amscd,amstext}
\usepackage{hyperref}
\usepackage{mathabx}
\usepackage{color}
\usepackage{epstopdf}
\usepackage{caption}
\usepackage{multicol}
\usepackage{cite}
\usepackage{amsthm}
\usepackage{mathtools}
\usepackage{bigints}
\usepackage{amsrefs}
\usepackage{enumitem}
\usepackage{mathrsfs}
\usepackage{tikz}
\usepackage{empheq}
\usepackage{framed}
\usepackage{ulem}
\usepackage{nameref}
\usepackage{esint}

\numberwithin{equation}{section}

\newtheorem{refthm}{Theorem}

\newcommand{\ra}{\rightarrow}

\newtheorem{theorem}{Theorem}
\newtheorem{lemma}{Lemma}

\newtheorem{corollary}{Corollary}
\newtheorem{definition}{Definition}
\newtheorem{proposition}{Proposition}

\newtheorem*{theorem*}{Theorem}
\newtheorem*{lemma*}{Lemma}
\newtheorem*{conj*}{Conjecture}
\newtheorem*{corollary*}{Corollary}
\newtheorem*{proposition*}{Proposition}
\newcommand{\rom}[1]{\uppercase\expandafter{\romannumeral #1\relax}}

\newcommand{\lan}{\langle}
\newcommand{\ran}{\rangle}

\newcommand{\Z}{\mathbb{Z}}
\newcommand{\R}{\mathbb{R}}
\newcommand{\C}{\mathbb{C}}

\newcommand{\N}{\mathbb{N}}

\newcommand{\cS}{\mathcal{S}}

\newcommand{\Ga}{\alpha}

\newcommand{\Ge}{\varepsilon}
\newcommand{\Gg}{\gamma}

\DeclareMathOperator*{\supp}{supp}

\makeatletter
\newcommand{\labitem}[2]{%
\def\@itemlabel{\textbf{#1}}
\item
\def\@currentlabel{#1}\label{#2}}
\makeatother

\theoremstyle{remark}
\newtheorem*{remark}{Remark}
\title[On the kernel conditions of operators mapping atoms to molecules in $h^p$]{On the kernel conditions of operators mapping atoms to molecules in local Hardy spaces}

\author{Chun Ho Lau}
\address{Department of Mathematical Science, University of Cincinnati, Cincinnati, OH 45221-0025, USA}
\email{lauco@ucmail.uc.edu}

\author{Claudio Vasconcelos}
\address{Laboratoire de Math\'ematiques d'Orsay, CNRS UMR 8628, Universit\'e Paris-Saclay, B\^atiment 307, 91405 Orsay Cedex, France}
\email{claudio.vasconcelos@alumni.usp.br}



\thanks{}

\begin{document}

\begin{abstract}

In this paper, we explore the relationship between the operators mapping atoms to molecules in local Hardy spaces $h^p(\R^n)$ and the size conditions of its kernel. In particular, we show that if the kernel of a Calder\'on--Zygmund-type operator satisfies an integral-type size condition and a $T^*-$type cancellation, then the operator maps $h^p(\R^n)$ atoms to molecules. On the other hand, assuming that $T$ is an integral type operator bounded on $L^2(\R^n)$ that maps atoms to molecules in $h^p(\R^n)$, then the kernel of such operator satisfies the same integral-type size conditions. We also provide the $L^1(\R^n)$ to $L^{1,\infty}(\R^n)$ boundedness for such operators connecting our integral-type size conditions on the kernel with others presented in the literature.
\end{abstract}

\maketitle

\section{Introduction}

The real variable theory of Hardy spaces can be traced back to the work of Fefferman and Stein in \cite{FeffStein}, in which they provided different characterizations of the real Hardy space, denoted by $H^p(\R^n)$ for $p>0$, using maximal functions and Poisson integrals. These spaces coincide with Lebesgue spaces when $p>1$ and for some applications it provides a suitable substitute for $L^p(\R^n)$ when $0<p\leq 1$ since Hardy spaces have nontrivial dual characterizations. Furthermore, when $0<p\leq 1$, tempered distributions belonging to $H^p(\R^n)$ can be decomposed in terms of bounded compactly supported functions $a$ satisfying vanishing moment conditions, that is
\begin{equation} \label{vanishing-moment}
\int_{\R^n} a(x)x^{\alpha} dx = 0, \quad \forall \, |\alpha|\leq N_p := \lfloor n(1/p-1) \rfloor.
\end{equation}
This decomposition was established for $H^p(\R)$ by Coifman \cite{Coifman} and generalized for any dimension by Latter \cite{Latter}. It turned out to be very useful, since many properties of the space can be reduced by studying it over the atoms. For instance, it is well known that the boundedness of Calder\'on-Zygmund type operators can be established by proving that the image of the operator by an atom is uniformly bounded in the $H^p-$norm. Since atoms have themselves uniformly bounded norm in $H^p(\R^n)$, one may try to show that the operator maps atoms into atoms, which is not the case since it may not preserve the compact support. This led to the study of a more general decomposition of Hardy spaces, called molecular decomposition, in which the hypothesis of compact support is relaxed.

The first discussion on molecular decompositions of Hardy spaces, using functions that are no longer compactly supported but satisfy certain Lebesgue norm size conditions and vanishing moments, was established by Taibleson and Weiss \cite{TaibWeiss}. 
This new decomposition were extensively used over the years to show the boundedness of certain classes of linear operators on $H^p(\R^n)$, since its boundedness follows by showing that the operator maps atoms into molecules.

Even though the classical Hardy spaces are useful, they still have some inconveniences. For instance, due to the cancellation conditions inherent in the Hardy space, they are not stable under multiplications by cutoff functions and the class of Schwartz functions $\mathcal{S}(\R^n)$ is not contained in $H^p(\R^n)$. Motivated by this, in \cite{Goldberg1979} Goldberg introduced a localized version of Hardy spaces, called local Hardy spaces and denoted by $h^p(\R^n)$, for $p>0$. As desired, $\mathcal{S}(\R^n)\subset h^p(\R^n)$ for all $0<p <\infty$ and $h^p(\R^n)$ is stable under multiplication by a smooth cutoff function. This localization property allows the extension of local Hardy spaces in different settings as manifolds. From a comparison between $H^p(\R^n)$ and $h^p(\R^n)$ (see \cite{Goldberg1979}*{Lemma 4}) one can get an analogous atomic decomposition for $h^p(\R^n)$, except that vanishing moment conditions \eqref{vanishing-moment} are required only for atoms supported in balls with small radii, which we call local vanishing moment condition. 

It turned out that this local vanishing moment condition is not necessarily required. For instance, in \cite{GaliaThesis}, Dafni introduced atoms in $h^p(\R^n)$, for $0<p\leq 1$, replacing the local vanishing moment condition by a control of the absolute value of their moments by the factor $r^{\beta}$, where $r$ is the radius of the ball in which the support of the atom is contained and $\beta>0$. Later, Komori \cite{Komori} showed that for $\frac{n}{n+1}<p<1$, one just needs to bound the absolute value of the moment by a constant; however, this is not the case when $p=\frac{n}{n+k}$ for some $k\in\N \cup \{ 0 \}$ (see the \cite[Example 3.4]{DLPV1}). For the particular case $p=1$, Dafni and Yue in \cite{DafniYue} showed an atomic decomposition for $h^1(\R^n)$ with atoms with moments being bounded above by $[\log(1+r^{-1})]^{-1}$. Later, motivated by the boundedness of inhomogeneous Calder\'on-Zygmund operators in $h^p(\R^n)$, the authors in \cite{DLPV1} extended these ideas and introduced approximate atoms and molecules for all $0<p\leq 1$. In particular, they showed an atomic decomposition with atoms satisfying the inhomogeneous cancellation condition
\begin{equation} \label{eqn:approx_cancel} 
     \displaystyle \left| \int_{B(x_B,r)}{a(x) (x-x_B)^{\alpha}dx} \right| \leq \left\{ \begin{array}{ll} C &\quad \text{if } |\alpha|<\gamma_p, \\ 
			& \\
			 \left[\log \left( 1+\dfrac{C}{r} \right)\right]^{-\frac{1}{p}} &\quad \text{if }|\alpha|=N_p = \gamma_p,
		\end{array} \right.
        \end{equation}
    where $C>0$ is a constant and $B(x_B,r)$ is the smallest ball containing the support of $a$.

To be more precise of operators we mentioned before, a Calder\'on--Zygmund singular integral operator is an $L^2$-bounded operator formally given by 
$$
Tf(x)=\int_{\R^n}K(x,y)f(y)dy, \quad \forall \, x\notin \supp(f)
$$ 
where the kernel $K$ is a continuous function away from the diagonal satisfying certain size conditions. It is well known that if $K$ satisfies the H\"ormander condition, \textit{i.e.}, 
$$
\sup_{y\in B}\int_{(2B)^c}|K(x,y)-K(x,c(B))|dx<\infty, \quad \text{for all balls $B$},
$$ 
then $T$ is bounded on $L^p(\R^n)$ for $1<p<2$, bounded from $L^1(\R^n)$ to $L^{1,\infty}(\R^n)$ and from $H^1(\R^n)$ to $L^1(\R^n)$, see \cite{SteinHA}*{Chapter I Section 5 and Chapter III Section 3.1}. However, H\"ormander condition is not enough to conclude that $T$ is bounded on $H^1(\R^n)$, even if $T^*(1)=0=T(1)$, as shown in \cite{YYD}. The boundedness of $T$, under the condition that $T^*(1)=0$, on $H^1(\R^n)$ is guaranteed considering kernels satisfying the classical H\"older regularity
\begin{align} \label{sizeA}
    |K(x,y)|\leq C|x-y|^{-n} \ \ \text{and} \ \ |K(x,y)-K(x,z)|\leq C \frac{|y-z|^{\delta}}{|x-y|^{n+\delta}} 
\end{align}
whenever $2|y-z|\leq |x-y|$, or some other weaker integral H\"ormander-type conditions presented in \cite{PV}. However, such operator $T$ associated with $K$ would not necessarily be bounded on $h^1(\R^n)$ nor from $h^1(\R^n)$ to $L^1(\R^n)$ (consider, for instance, the Hilbert transform on $\R$ and the characteristic function on $[0,1]$). An alternative for local Hardy spaces is to consider inhomogeneous Calder\'on-Zygmund singular integral operators, introduced by \cite{DHZ}. The main difference between this latter and the classical one is that the kernel of the former has a stronger decay far away from the diagonal, namely 
\begin{equation} \label{inhomogeneous-kernel}
    |K(x,y)|\leq C\min\lbrace|x-y|^{-n},|x-y|^{-n-\mu}\rbrace, \quad \forall \, x \neq y \text{ and some $\mu>0$}.
\end{equation}
Hence, the class of inhomogeneous  Calder\'on--Zygmund singular integral operators is smaller than the class of Calder\'on--Zygmund singular integral operators. This kind of operator includes some of the pseudo-differential operators, see \cite{DLPV1,DHZ}. Then, if $T$ is an inhomogeneous Calder\'on--Zygmund singular integral operator, the authors in \cite{DLPV1,DLPV2} found necessary and sufficient conditions for the boundedness of such operators on $h^p(\R^n)$ for all $0<p\leq 1$, improving and extending the work \cite{DHZ}.

The core of proving the boundedness of a Calder\'on--Zygmund singular integral operator $T$ on Hardy spaces or local Hardy spaces is to show that $T$ maps atoms to molecules. This strategy works perfectly in $H^p(\R^n)$, under the assumption that the operator satisfies $T^*(x^\alpha)=0$, i.e.,
$$
\int_{\R^n} Ta(x)x^{\alpha}dx=0, \quad \forall f\in L^2_{N_p}(\R^n) \text{ and } |\alpha|\leq N_p.
$$
This condition is actually necessary and sufficient for the boundedness of such operators in $H^p(\R^n)$. However, a local version of such cancellation condition on $T$ is too strong in the local setting (for instance, for pseudodifferential operators). The key to find less restrictive necessary and sufficient conditions for the boundedness of inhomogeneous Calder\'on-Zygmund operators in the local setting, was to consider approximate atoms and molecules. To do this, if $a$ is an atom in $h^p(\R^n)$, we show that only the conditions imposed on the kernel imply that $Ta$ satisfy the size conditions of the molecule. Then, the additional cancellation condition that $T^*((\cdot-c)^{\alpha})$ satisfy a local Campanato-type estimate will show the approximate moment condition of the molecule (see \cite{DLPV1}*{Section 5}). On the other hand, if we suppose that $Ta$ satisfies the size condition of molecules, then $T^*((\cdot-c)^{\alpha})$ must satisfy the local Campanato type conditions, see \cite{DLPV2}*{Theorem 2}. In a more general setting, the authors in \cite[Theorem 1]{DLPV2} also proved the necessity of this local Campanato condition for genereral linear and bounded operators in $h^p(\R^n)$ that maps atoms into \textit{pseudo-molecules}, a generalization of \textit{pre-molecules}, that will be defined in Definition \ref{def:molecules}. 

Motivated by these results, the goal of this paper is to further investigate this relation between the size conditions of the kernel and the property of the operator to map atoms to molecules. This work is inspired by a related question answered in the setting of the homogeneous Triebel-Lizorkin spaces $\dot{F}^{\alpha,q}_p(\R^n)$ for $\alpha\in (0,1)$ and $1<p,q<\infty$ in \cite{FHJW}, which is related to Hardy spaces since $H^p(\R^n) = \dot{F}^{0,2}_{p}(\R^n)$ and $h^p(\R^n) = {F}^{0,2}_{p}(\R^n)$ for all $0<p\leq 1$, where $F^{\alpha,q}_{p}(\R^n)$ denotes the inhomogeneous Triebel-Lizorkin spaces. In particular, the authors showed in Theorem 1.16 that a linear operator $T$ which is continuous from $C^{\infty}_c(\R^n)$ with mean 0 to the distribution spaces mapping certain smooth atoms to smooth molecules, then the Schwartz kernel of the restriction on $C^{\infty}_c(\R^n)$ of the extension of $T$ on $\dot{F}^{\alpha,q}_p(\R^n)$ indeed satisfies the classical Calder\'on-Zygmund estimates $|K(x,y)|\leq C|x-y|^{-n}$ and  $|K(x,y)-K(x,z)|\leq C|y-z|^{\delta}|x-y|^{-n-\delta}$ whenever $2|y-z|\leq |x-y|$. However, the necessary conditions on the kernel of an operator to map atoms to molecules for (local) Hardy spaces are unclear.

To motivate our main result, we start by showing that under a suitable integral-type condition on the kernel and some cancellation condition on $T$, then $T$ maps atoms to approximate molecules.

\begin{theorem} \label{mainthm0}
    Let $0<p\leq 1\leq q \leq 2$, $1\leq s\leq 2$  with $p\neq s$, and $T$ be a Calder\'on-Zygmund singular integral operator. If the kernel of $T$ satisfies
    \begin{itemize}
        \item[(i)] there exist a collection of polynomials $\lbrace P_{K,B}(x,y)\rbrace_{B}$  in $y$-variables (of degree at most $N_p$), a constant $\varepsilon>0$, and a constant $C>0$ such that for all balls $B:=B(c,r)$ we have
\begin{align} \label{Kernel1}
    \bigg\|\bigg(\frac{|x-c|}{r}\bigg)^{n(\frac{1}{p}-\frac{1}{s})+\frac{\varepsilon}{s}}\|K(x,y)-P_{K,B}(x,y)\|_{L^{q}(B,dy/|B|)}\bigg\|_{L^s([B(c,2r)]^c,dx)}\leq C r^{n(s^{-1}-1)},
\end{align}
where $P_{K,B}(x,y)\equiv 0$ if $r\geq 1$;
    \item[(ii)] for all $|\alpha|\leq \lfloor n(\frac{1}{p}-1)\rfloor$,  $c\in\R^n$ and $0<r<1$, $f=T^*((\cdot-c)^{\alpha})$, formally defined by $\langle f, a\rangle = \lan (x-c)^{\Ga}, T(a)\rangle$, satisfies 
\begin{align} \label{Tstar}
    \bigg(\fint_{B(c,r)}|f(x)-P^{N_p}_B(f)(x)|^qdx\bigg)^{\frac{1}{q}}\leq C \Psi_{p,\alpha}(r),
\end{align}
where $P^{N_p}_{B}(f)(x)$ is the polynomial of degree $\leq N_p$ with the same moments as $f$ over $B$ up to order $N_p$, and
$$
\displaystyle{\Psi_{p,\alpha}(t):=\left\{ \begin{array}{ll} 
t^{\gamma_p} &\quad \text{if } |\alpha|< \gamma_p, \\
 t^{\gamma_p} \left[\log \left( 1+\dfrac{C}{t} \right)\right]^{-\frac{1}{p}} &\quad \text{if } |\alpha| = \gamma_p = N_p \in \Z_{+};
			\end{array} \right.}
$$
\end{itemize}
then $T$ maps $(h^p,q')$ exact atoms supported in $B(c,r)$ to $(h^p,s,n(\frac{s}{p}-1)+\varepsilon, C')$ approximate molecules centered in $B(c,2r)$ for some $C'>0$. Thus, $T$ is bounded on $h^p(\R^n)$.
\end{theorem}

One can verify that a kernel $K$ satisfies \eqref{Kernel1} if $K$ satisfies \eqref{inhomogeneous-kernel}  and either \eqref{sizeA} or
$$
\bigg(\int_{2^{j}r\leq |x-c|<2^{j+1}r}|K(x,y)-K(x,c)|^sdx\bigg)^{1/s}\leq C_{n,s} (2^j)^{\frac{1}{s}-1-\delta}r^{\frac{1}{s}-1},
$$ 
for every $|y-c|\leq r$, introduced in \cite{DLPV1}, with $P_{K,B(z,r)}(x,y)=K(x,z)$ if $0<r<1$. In addition, the kernel condition given in \cite{PV}*{Equation (4.7)} and \cite{LThesis}*{Equation (35)} imply \eqref{Kernel1}. Condition \eqref{Kernel1} is inspired by a BMO-type kernel condition first introduced by Suzuki in \cite{Suzuki1}, and can be regarded as a local-Campanato-type kernel condition with weights in the $x$-variable. In fact, Suzuki showed that Calder\'on-Zygmund operators associated to such kernels are bounded from $H^1(\R^n)$ to $L^1(\R^n)$ but they are not necessarily bounded from $L^{1}(\R^n)$ to $L^{1,\infty}(\R^n)$ (see \cite{Suzuki1}*{Theorem 2}).

Next, we state the following converse of Theorem \ref{mainthm0}, which is the main theorem of this paper.

\begin{theorem} \label{mainthm}
    Let $0<p\leq 1\leq q\leq s\leq 2$ ($p\neq s$),  and $\varepsilon >0$. Suppose $Tf(x)= \int K(x,y)f(y)dy$ whenever $x\notin\supp(f)$ is a bounded linear operator on $L^2(\R^n)$. If there exists $C>0$ such that $T$ maps $(h^p, q')$ exact atoms to $(h^p,s,n(\frac{s}{p}-1)+\varepsilon, C)$ approximate molecules with respect to the same ball, then for all balls $B=B(c,r)$ the kernel $K$ satisfies \eqref{Kernel1},
where 
$$P_{K,B}(x,y)=P_{K,B(c,r)}(x,y)=\begin{dcases}
       \sum_{|\alpha|\leq N_p}\frac{1}{\alpha!}F_{\alpha,c,r}(x)(y-c)^{\alpha}, &\quad \text{if }0<r<1, \\
    0 &\quad \text{otherwise},
    \end{dcases}$$
and for some functions $F_{\alpha,c,r}(x)$ on $|x-c|\geq r$. If $N_p\neq \gamma_p$ is additionally assumed, then $F_{\alpha,c,r}\in h^p(\R^n)$ and can be written as $F_{\alpha,c,r}=\sum_{j=0}^{\infty} \lambda_j M_{j,c}$, where $\sum_{j=0}^{\infty}|\lambda_j|^{p}<\infty$ and $M_{j,c}$ are $h^p(\R^n)$ approximate molecules with respect to $B(c,2^{-j+1})$.
\end{theorem}

In Remark \ref{remark_end} we make several comments on the hypothesis of this theorem.

As a corollary of the previous theorem, we can show together with \cite{Suzuki0}*{Theorem 1}, that if the operator maps atoms to molecules, then it can be extended continuously to an operator from $L^{1}(\R^n)$ to $L^{1,\infty}(\R^n)$. 

\begin{corollary} \label{p=1case}
    Let $\frac{n}{n+1}<p\leq 1\leq q\leq s \leq 2$ with $p\neq s$ and $\varepsilon>0$. Suppose $Tf(x)= \int K(x,y)f(y)dy$ whenever $x\notin\supp(f)$ is a bounded linear operator on $L^2(\R^n)$, and $T$ maps $(h^p, q')$ exact atoms to  $(h^p,s,n(\frac{s}{p}-1)+\varepsilon, C)$ approximate molecules with respect to the same ball $B(c,r)$ for some $C>0$, then $K$ satisfies \eqref{Kernel1} with $P_{K,B(c,r)}(x,y)=K(x,c)$ if $r< 1$ (and 0 when $r\geq 1$) and 
    \begin{equation} \label{suzuki-condition}
    \sup_{B}\int_{(4B)^c} \fint_{B} \big|K(x,y)-K(x,c)\big| dydx <\infty.
    \end{equation}
     Moreover, $T$ can be extended to a bounded linear operator from $L^1(\R^n)$ to $L^{1,\infty}(\R^n)$.
\end{corollary}

\begin{remark}
    We shall mention that Condition \eqref{suzuki-condition} is equivalent to H\"{o}rmander's condition  by \cite{Suzuki1}*{Theorem 1}.
\end{remark}

g\section{Background and notation}

Throughout this paper, we denote by $q'$ the H\"older conjugate of $q$, \textit{i.e.} $q^{-1}+(q')^{-1}=1$, and $B(c,r)$ the ball in $\R^n$ centered at $c\in\R^n$ with radius $r>0$. Also, the constants $C$, $C'$, etc., in the proofs may vary from line to line, and we use the subscript to highlight the dependence of the variables in the constants. By $\N$ we mean the positive integers.

We start with the formal maximal definition of the Hardy space $h^p(\R^n)$.

\begin{definition}
   Let $p\in (0,\infty)$. We say that the tempered distribution $f\in \mathcal{S}'(\R^n)$ belongs to the local Hardy space $h^p(\R^n)$ if there exists $\Phi\in \mathcal{S}(\R^n)$ satisfying $\int \Phi \neq 0$ such that 
$$\mathcal{M}_{\Phi}f(x) = \sup_{0<t<1} |\Phi_t\ast f(x)|\in L^p(\R^n),$$
 where $\Phi_t(x):= t^{-n}\Phi(xt^{-1})$.
\end{definition}
\noindent The functional $\|f\|_{h^p}= \big\|\mathcal{M}_{\Phi}f\big\|_{L^p}$ defines a quasi-norm in $h^p(\R^n)$ when $0<p<1$ and a norm otherwise. We may refer it as a norm for convenience. The associated distance $d(f,g)=\|f-g\|_{h^p}^p$ defines a metric on $h^p(\R^n)$. Next we provide the definition of $(h^p,q)$ atoms.
\begin{definition}
    Let $0<p\leq 1 \leq q\leq \infty$ with $p\neq q$. We say the function $a$ is an $(h^p,q)$ exact atom if there exists a ball $B=B(c,r)$ such that 
\begin{enumerate}[label=(\arabic{*}), ref=(\arabic{*})]
    \item $\supp(a)\subset B$,
    \item $\|a\|_{L^q(B)}\leq |B|^{\frac{1}{q}-\frac{1}{p}}$, and 
    \item $\displaystyle \int_{B} a(x)(x-c)^{\alpha}dx =0$ for all $|\alpha|\leq N_p$ if $0<r<1$, 
\end{enumerate}
where $N_p:=\lfloor n(p^{-1}-1)\rfloor$, the largest possible integer that is at most $\gamma_p:=n(p^{-1}-1)$. If Condition $(3)$ is replaced by \eqref{eqn:approx_cancel}, we call it an $(h^p,q)$ approximate atom.
\end{definition}

Meanwhile, the molecular theory for $h^p(\R^n)$  with non-exact cancellation is studied in \cite{Komori} for $\frac{n}{n+1}<p<1$ and $h^1(\R)$ in \cite{DafniLif}. In \cite{DLPV1}, the authors studied the approximated molecular theory for $h^p(\R^n)$ for all $p\leq 1$. We now give the definitions of approximate molecules.

\begin{definition} \label{def:molecules}
    Let $0<p\leq 1 \leq q< \infty$ with $p\neq q$ and $\lambda>n(\frac{q}{p}-1)$. We say that a function $M$ is an $(h^p,q,\lambda, C)$ approximate molecule if there exist a ball $B=B(c,r)$ and $C>0$ (independent of $M$) such that 
\begin{enumerate}[label=(\arabic{*}), ref=(\arabic{*})]
    \item[$(M_1)$] $\|M\|_{L^q(B)}\leq Cr^{n(\frac{1}{q}-\frac{1}{p})}$,
    \item[$(M_2)$] $\|M|\cdot-c|^{\frac{\lambda}{q}}\|_{L^q(B^c)}\leq C r^{\frac{\lambda}{q}+n(\frac{1}{q}-\frac{1}{p})}$, and 
    \item[$(M_3)$]  $ \displaystyle \left| \int_{\R^n}{M(x) (x-c)^{\alpha}dx} \right| \leq \left\{ \begin{array}{ll}  C, &\quad \text{if } |\alpha|<\gamma_p, \\ 
	& \\
	\left[\log \left( 1+\dfrac{1}{C r} \right)\right]^{-\frac{1}{p}}, &  \quad \text{if }|\alpha|=N_p = \gamma_p.
		\end{array} \right. $
\end{enumerate}

Following \cite{DLPV2}*{Definition 3}, when the function $M$ satisfies only the size conditions ($M_1$) and ($M_2$), we call it a pre-molecule.
\end{definition}
We shall note that it is not harmful to change $B$ in ($M_1$) and ($M_2$) by $kB$ for some $k>1$ (independent of $B$) or to $\R^n$. Moreover, the constant $\lambda$ can be written as $\lambda=n(\frac{q}{p}-1)+\varepsilon$ for some $\varepsilon>0$ and the upper bound of (2) becomes $Cr^\frac{\varepsilon}{q}$.

In this paper, we only use the decomposition of $h^p(\R^n)$ in terms of exact atoms and approximate molecules. More precisely, we will use the following.

\begin{refthm}[\cites{Goldberg1979,DLPV1}] \label{atomicdecompo}
    Let $0<p\leq1\leq q\leq \infty$ with $p<q$ and $\lambda >n(\frac{q}{p}-1)$. Then the following are equivalent. 
    \begin{enumerate}
        \item The distribution $f\in \mathcal{S}'(\R^n)$ is in $h^p(\R^n)$.
        \item There exist a sequence $\lbrace \lambda_{j}\rbrace_{j\in\N} \in \ell^p(\C)$ and a sequence of $(h^p,q)$ exact atoms $\lbrace a_{j}\rbrace_{j\in \N}$ such that $f=\sum_{j\in \N} \lambda_j a_j$ in $\mathcal{S}'(\R^n)$ and in $h^p(\R^n)$, with $$\|f\|_{h^p} \simeq \inf \bigg\{ \bigg(\sum_{j\in \N}|\lambda_j|^p\bigg)^{1/p} \bigg\},$$ where the infimum is taken over all such atomic representations.
        \item There exist a sequence of $\lbrace \lambda_{j}\rbrace_{j\in\N}\in \ell^p(\C)$ and a sequence of $(h^p,q, \lambda, C)$ approximate molecules $\lbrace M_{j}\rbrace_{j\in \N}$ such that $f=\sum_{j\in \N} \lambda_j M_j$ in $\mathcal{S}'(\R^n)$ and in $h^p(\R^n)$, and $\|f\|_{h^p} \simeq \inf \left\{ \left(\sum_{j\in \N}|\lambda_j|^p\right)^{1/p} \right\}$, where the infimum is taken over all such molecular representations.
    \end{enumerate}
\end{refthm}
\section{Proofs}

In this section we provide the proof of the theorems.

\subsection{Proof of Theorem \ref{mainthm0}}

We start by showing that the pairing $\langle T^*((\cdot-c)^{\Ga}, a\rangle$ is well defined for kernels satisfying conditions of Theorem  \ref{mainthm0}.

\begin{proposition}
    Let $0<p\leq 1$ and $1\leq s\leq 2$. Fix $T$ the operator given in Theorem \ref{mainthm0}. Let $g\in L^{q'}(\R^n)$ ($q'\geq 2$) that has support in $B=B(c,r)$ and $g$ such that $\int_{\R^n}(x-c)^{\Ga}g(x)dx =0$ for all $|\alpha|\leq N_p$. Then 
    $$
    \int_{\R^n}|x-c|^{\Ga}|T(g)(x)|dx <\infty \quad \text{for all $|\alpha|\leq N_p$.}
    $$ 
\end{proposition}
\begin{proof}
    Note that from the $L^2-$boundedness of $T$ we have
    $$\int_{2B}|x-c|^{\alpha}|T(g)(x)|dx \leq C_n r^{\alpha+\frac{n}{2}}\|Tg\|_{L^2}\leq  C_n r^{\alpha+\frac{n}{q}}\|T\|_{L^2\to L^2}\|g\|_{L^{q'}}<\infty.$$
    It remains to estimate the integral on $(2B)^c$. Note that, using the cancellation condition of $g$, 
    \begin{align*}
        &\int_{(2B)^c} |x-c|^{\alpha} |T(g)(x)|dx \leq \int_{(2B)^c} |x-c|^{|\alpha|} \int_{B}|K(x,y)-P_{K,B}(x,y)||g(y)|dy dx \\
        &\leq |B|^{\frac{1}{q'}}\|g\|_{L^{q'}} \int_{(2B)^c} |x-c|^{\alpha} \|K(x,y)-P_{K,B}(x,y)\|_{L^q(B, dy/|B|)} dx \\
        &\leq |B|^{\frac{1}{q'}}\|g\|_{L^{q'}}  \bigg\||x-c|^{n(\frac{1}{p}-\frac{1}{s})+\frac{\varepsilon}{s}}\|K(x,y)-P_{K,B}(x,y)\|_{L^{q}(B,\frac{dy}{|B|})}\bigg\|_{L^s((2B)^c,dx)} \\
        & \quad \quad \quad \quad\quad \times \| |\cdot-c|^{-n(p^{-1}-s^{-1})-\frac{\varepsilon}{s}+|\alpha|}\|_{L^{s'}((2B)^c)}.
    \end{align*}
    Note that $-n(s')(\frac{1}{p}-\frac{1}{s})-\frac{s'\varepsilon}{s}+s'|\alpha|+n=-ns'(\frac{1}{p}-1)-\frac{s'\varepsilon}{s}+s'|\alpha|<0$ if and only if $|\alpha|\leq n(p^{-1}-1)+\frac{\varepsilon}{s}$, which is the case because $|\alpha|\leq N_p$. Therefore, the above estimate is bounded by a function of $r$ and is therefore finite.
\end{proof}

Although the proof of Theorem \ref{mainthm0} is similar to the one in \cite{DLPV1}*{Theorem 5.3}, for the sake of completeness, we will give the proof here.

 \begin{proof}[Proof of Theorem \ref{mainthm0}]\ \\
Let $a$ be an $(h^p,q')$ exact atom supported in $B=B(c,r)$. We show that $T(a)$ is an $(h^p,s,n(\frac{s}{p}-1)+\varepsilon)$ approximate molecule centered in $2B$. 

By the $L^2$-boundedness of $T$, we have 
$$
\|Ta\|_{L^s(2B)}\leq \|Ta\|_{L^2(2B)}|2B|^{\frac{1}{s}-\frac{1}{2}}\leq C_{q}\|T\|_{L^2\rightarrow L^2} |B|^{\frac{1}{2}-\frac{1}{p}+\frac{1}{s}-\frac{1}{2}}=C_{T,q} |B|^{\frac1s-\frac1p}. 
$$

To show ($M_2$), we split $r$ into two cases. For $r< 1$, using the exact cancellation condition on $a$ and the hypothesis,
\begin{align*}
    &\|Ta|\cdot-c|^{n(\frac{1}{p}-\frac{1}{s})+\frac{\varepsilon}{s}}\|_{L^s(2B)^c}= \bigg\||x-c|^{n(\frac{1}{p}-\frac{1}{s})+\frac{\varepsilon}{s}} \int_{B} [K(x,y)-P_{K,B}(x,y)]a(y)dy \bigg\|_{L^s((2B)^c,dx)} \\
    &\quad\quad \leq \bigg\||x-c|^{n(\frac{1}{p}-\frac{1}{s})+\frac{\varepsilon}{s}} \int_{B} |K(x,y)-P_{K,B}(x,y)||a(y)|dy \bigg\|_{L^s((2B)^c,dx)}   \\
    &\quad\quad\leq  \bigg\||x-c|^{n(\frac{1}{p}-\frac{1}{s})+\frac{\varepsilon}{s}}  \|K(x,y)-P_{K,B}(x,y)\|_{L^q(B,dy/|B|)}|B|^{\frac{1}{q}+\frac{1}{q'}-\frac{1}{p}} \bigg\|_{L^s((2B)^c,dx)} \\
    &\quad\quad\leq C |B|^{1-\frac{1}{p}} r^{n(\frac{1}{p}-\frac{1}{s})+\frac{\varepsilon}{s}+n(s^{-1}-1)} = C_{n,p,q,s,T} \, r^{\frac{\varepsilon}{s}}.
\end{align*}
When $r\geq 1$, the argument is exactly the same except $P_{K,B}(x,y)=0$ in this case.

Equation \eqref{Tstar} implies that $Ta$ satisfies ($M_3$) using H\"older's inequality and the fact that $$\int (x-c)^{\alpha}Ta(x)dx = \int T^*[(\cdot-c)^{\alpha}](y) a(y)dy. $$
 \end{proof}

\subsection{Proof of Theorem \ref{mainthm}}
The first main step in the proof of Theorem \ref{mainthm} is to describe a suitable decomposition of the kernel of $T$, following the analogous ideas presented in \cite{FHJW}*{Theorem 1.16} and \cite{SteinHA}*{pp. 244}. 

\begin{lemma} \label{lemma:kernel-decomposition}
There exists a sequence $\{ \psi_j \}_{j} \subset C^{\infty}_c(\R^n) $ of radial functions supported in $B(0,1)$ such that $\int \psi_0 = 1$, $\int \psi_j = 0$ for all $j\geq 1$ and
\begin{equation} \label{eq:approx}
\sum_{j=0}^{\infty}\widehat{\psi_j}(\xi)=1.
\end{equation} 
Moreover, if $K_j$ denotes the kernel of $T(\cdot \ast \psi_j)$, then $K_j(x,y)= T(\psi_j(\cdot-y))(x)$ and the kernel of the operator $T|_{C^{\infty}_c(\R^n)}$ is given by $K(x,y) = \sum_{j=0}^{\infty}K_j(x,y)$ in $\mathcal{S}'(\R^n\times \R^n).$
\end{lemma}

\begin{proof}
We start by constructing the functions $\{\psi_j\}_j$. Let $\Phi$ be a non-negative function in $C^{\infty}_c(\R^n)$ with $\supp(\Phi)\subset B(0,{1}/{2})$, $\int \Phi =1$ and $\int \Phi(x)x^{\alpha}dx \neq 0$ for $1\leq |\alpha| \leq 2N_p$. Define 
$$
\psi_0(x)=c_0\Phi(x)-\sum_{1\leq |\alpha|\leq N_p} c_{\alpha}x^{\alpha}\Phi(x),
$$ 
for some suitable constants $c_{\alpha}\in \R$ such that $\int x^{\alpha}\psi_0(x)dx=0$ for all $1\leq |\alpha|\leq N_p$ and $\int \psi_0(x)dx=1$. Then, define $\psi(x)=2^n\psi_0(2x)-\psi_0(x)$ and let $\psi_j(x)=2^{(j-1)n}\psi(2^{j-1}x)$ for any $j\in\N$. By the translation property of the Fourier transform we get $\widehat{\psi_j}(\xi)=\widehat{\psi_0}(2^{-j}\xi)-\widehat{\psi_0}(2^{1-j}\xi)$. In this way, 
\begin{align*}
    \sum_{j=0}^{N}\widehat{\psi_j}(\xi) =\widehat{\psi_0}(2^{-N}\xi)\rightarrow \widehat{\psi_0}(0)=1 \ \ \text{as $N\rightarrow \infty.$}
\end{align*}
For $f\in C^{\infty}_c(\R^n)$ we have $f = \sum_{j=0}^{\infty} f\ast \psi_j$ in $\mathcal{S}(\R^n)$, and from \eqref{eq:approx} and Plancherel identity it also holds in $L^2(\R^n)$. Moreover,
\begin{equation} \label{eq:split-T}
T(f)= \sum_{j=0}^{\infty}T(f\ast \psi_j) \quad \text{in $\mathcal{S}'(\R^n)$ and $L^2(\R^n)$.}
\end{equation}  
If $K_j$ denotes the kernel of $T(\cdot \ast \psi_j)$, simply by the integral representation of $T$ we have that $K_j(x,y)= T(\psi_j(\cdot-y))(x)$ and from \eqref{eq:split-T}, together with the $L^2$ continuity of $T$, the kernel of $T$ is given by $K = \sum_{j=0}^{\infty}K_j$ in $\mathcal{S}'(\R^n \times \R^n)$. Since $K$ is the kernel of $T$ and such decomposition is true in $\mathcal{S}'$, we have from \eqref{eq:split-T} 
\begin{align*}
\lan\lan K-\sum_{j=1}^{N}K_j, \phi\otimes f\ran\ran &=  \lan  T(f)-\sum_{j=1}^{N}T_j(f),\phi\ran\ra 0,
\end{align*}
where $\langle\langle\cdot,\cdot\rangle\rangle$ denotes the paring between $\mathcal{S}'(\R^{n}\times \R^n)$ and $\mathcal{S}(\R^{n}\times \R^n)$ and $\langle\cdot,\cdot\rangle$ denotes the paring between $\mathcal{S}'(\R^{n})$ and $\mathcal{S}(\R^{n})$. Since the tensor product is dense in $\mathcal{S}'(\R^n\times \R^n)$, we have the convergence in $\cS'(\R^n\times \R^n)$.
\end{proof}

Since we want to establish estimates for $K_j(x,y)$, in which the construction was described in the previous lemma, the next step is to provide additional information on the family $\{\psi_j\}_j$, which is useful in Proposition \ref{submainthm}.

\begin{lemma} \label{atomlemma}
    Let $0<p\leq1$ and $\{\psi_j\}_j \subset C^{\infty}_c(\R^n) $ the family of functions constructed in the proof of Lemma \ref{lemma:kernel-decomposition}. Then the following are true:
    \begin{enumerate}
        \item The function $x \mapsto \displaystyle \frac{\psi_0(x)}{\|\psi_0\|_{L^{\infty}} \, |B(0,2)|^{\frac1p}}$ is an $(h^p,\infty)$ exact atom supported in $B(0,2)$;
        \item For $N\in \N\cup\{0\}$, the function $x \mapsto \displaystyle \frac{\partial^{\beta} \psi_j(x)}{2^{(j-1)(|\beta|+n)} \, (2^n+1) \, \|\psi_0\|_{C^N} \, |B(0,2^{-j+1})|^{\frac{1}{p}}}$ is an $(h^p,\infty)$ exact atom supported in $B(0,2^{-j+1})$ for all $j\in \N$ and all multi-indices $\beta$ with $|\beta|\leq N$, where 
        $$\|\psi\|_{C^N}:=\sup_{x\in\R^n} \, \max_{|\beta|\leq N}|\partial^{\beta}\psi(x)|.$$
        \item If $|y-c|\leq \frac{2^{-j+1}}{3}$, then 
        \begin{align}\label{Ax}
            A_j(x) := \psi_j(x-y)-\psi_j(x-c)-\sum_{1\leq |\alpha|\leq N_p}\frac{(-1)^{|\alpha|}}{\alpha!} \left(\partial^{\alpha}\psi_j(x-c)\right)(y-c)^{\alpha},
        \end{align}
        is also a constant multiple (with constant $c_{p,n} \, 2^{j(N_p+1-\gamma_p)}|y-c|^{N_p+1}$) of $(h^p,\infty)$ atoms supported in $B(c, 2^{-j+2}).$
    \end{enumerate}
\end{lemma}
\begin{proof} 
The proof of (1) is immediate. We show the others.
\begin{itemize}
\item[(2)]By construction, $\supp(\psi)\subset B(0,1)$ and $\supp(\partial^{\beta}\psi_j)\subset B(0, 2^{-j+1})$. A direct calculation shows that 
\begin{align*}
    &\bigg\|\frac{\partial^{\beta}\psi_j}{2^{(j-1)(|\beta|+n)}(2^n+1)\|\psi_0\|_{C^N}|B(0,2^{-j+1})|^{\frac{1}{p}}}\bigg\|_{L^{\infty}} \\
    &\quad \quad\quad \quad \quad\quad \quad \quad\quad \quad \quad\quad \leq \frac{2^{(j-1)(|\beta|+n)}(2^n+1)\|\psi_0\|_{C^N}}{2^{(j-1)(|\beta|+n)}(2^n+1)\|\psi_0\|_{C^N}|B(0,2^{-j+1})|^{\frac{1}{p}}}\\
    &\quad \quad\quad \quad \quad\quad \quad \quad\quad \quad \quad\quad = |B(0,2^{-j+1})|^{-\frac{1}{p}}
\end{align*}
for any $j\in \N$ and $|\beta|\leq N$. To verify the cancellation condition, it suffices to check it for $\psi_j$. Suppose first $|\beta|=0$, and note that 
\begin{align*}
    \int_{\R^n} 2^{(j-1)n } \psi_j(x)dx = \int_{\R^n} \psi(x)dx = \int_{\R^n}[2^n\psi_0(2x)-\psi_0(x)] dx =0.
\end{align*}
For $1\leq |\alpha|\leq N_p$, we have 
\begin{align*}
    \int_{\R^n} x^{\alpha} 2^{(j-1)n}\psi_j(x)dx &= 2^{-(j-1)|\alpha|}   \int_{\R^n} x^{\alpha}\psi(x)dx \\
    &= 2^{-(j-1)|\alpha|} \int_{\R^n} (2^{-n|\alpha|}-1) x^{\alpha}\psi_0(x)dx \\
    &=0.
\end{align*}
For $1\leq |\beta|\leq N$, the cancellation conditions hold because 
\begin{align*}
    \int_{\R^n} x^{\alpha} \partial^{\beta}\psi_j(x)dx &= (-1)^{|\beta|}\int_{\R^n} \partial^{\beta}(x^{\alpha})\psi_j(x)dx =0 
\end{align*}
and $ \partial^{\beta}(x^{\alpha})$ is again a monomial with non-negative power.

\item[(3)] Let $0<|y-c|\leq \frac{2^{-j+1}}{3}$.  
Note that $\supp(A_j)\subset B(y,2^{-j+1})\cup B(c,2^{-j+1})\subset B(c,2^{-j+2})$. By Taylor's theorem, there exists $v=ty+(1-t)c$ with $t\in [0,1]$ such that 
\begin{align*}
    A_j(x) = \sum_{|\alpha|=N_p+1} \frac{(-1)^{|\alpha|}}{\alpha!} \big(\partial^{\alpha}[\psi_j(x-v)]\big) (y-c)^{\alpha}.
\end{align*}
Then, 
\begin{align*}
    \|A_j\|_{L^{\infty}} &\leq \sum_{|\alpha|=N_p+1} \frac{1}{\alpha!} c_n 2^{(j-1)(N_p+1+n)}\|\psi_0\|_{C^{N_p+1}}|y-c|^{N_p+1} \\
    &\leq c_{p,n}2^{(j-1)(N_p+1+n)}\|\psi_0\|_{C^{N_p+1}}|y-c|^{N_p+1}.
\end{align*}
Meanwhile, the cancellation conditions hold because of a change of variable and the fact that $\int x^{\beta}\partial^{\alpha}\psi_j=0$ for all $|\beta|\leq N_p$. Therefore, $A_j[c_{p,n} 2^{(j-1)(N_p+1+n)}\|\psi_0\|_{C^{N_p+1}}|y-c|^{N_p+1}|B(0,2^{-j+2})|^{p^{-1}}]^{-1}$ is an $(h^p,\infty)$ exact atom.
\end{itemize}\end{proof}

Finally, in the next proposition, we show some important estimates on the kernels $K_j$ when $T$ maps atoms into pre-molecules.

\begin{proposition} \label{submainthm}
    Let $0<p\leq 1\leq q\leq s\leq 2$ with $p\neq s$, and $0<\varepsilon<s(N_p+1-\gamma_p)$. Suppose $T$ is bounded on $L^2(\R^n)$. If there exists $C>0$ such that $T$ maps $(h^p, q')$ exact atoms to $\left(h^p,s,n\big(\frac{s}{p}-1\big)+\varepsilon, C\right)$ pre-molecules with respect to the same ball, then
\begin{align} \label{maingoal}
   \sum_{j=0}^{\infty}\bigg\|\bigg(\frac{|x-c|}{r}\bigg)^{n(\frac{1}{p}-\frac{1}{s})+\frac{\varepsilon}{s}}\|K_j(x,y)-P_{K_j,B}(x,y)\|_{L^{q}(B,dy/|B|)}\bigg\|_{L^s([B(c,4r)]^c,dx)}\leq C_{n,p,q,s} r^{n(s^{-1}-1)},
\end{align} 
where 
$$P_{K_j,B(c,r)}(x,y)=\begin{cases}
       \displaystyle \sum_{|\alpha|\leq N_p}M_{\alpha,j,c}(x)(y-c)^{\alpha}, &\quad \text{if }0<r<1, \\
    0 &\quad \text{otherwise},
    \end{cases}$$
and for some pre-molecule $M_{\alpha,j,c}$ centered at $c$. In particular, for fixed $B=B(c,r)$, $K-\sum_{j}P_{K_j,B}$ is a function on $\lbrace (x,y)\in \R^n\times \R^n: x\notin B(c,4r), \ y\in B(c,r)\rbrace$.
\end{proposition}

\begin{proof}

Let $0<p\leq 1$, $c\in \R^n$ and $r>0$. Without loss of generality, we consider the case $s=q$. By Lemma \ref{atomlemma}, since $\tilde{\psi}_j := C_{n,p} 2^{n(1-j)\left( 1-\frac{1}{p} \right)} \psi_j$ is an exact $(h^p,\infty)$ atom supported in $B(0,2^{1-j})$, we have by translation that $\tilde{\psi_j}(\cdot-y)$ is also an $(h^p,\infty)$ atom supported in $B(y,2^{1-j})$. Then, since by hypothesis $T$ maps atoms into pre-molecules,
for all $y \in \R^n$ we have
\begin{align} \label{Mole2}
    \int_{|x-y| \geq  2^{-j+1}} & |x-y|^{n\left(\frac{q}{p}-1\right)+\varepsilon} |T(\tilde{\psi_j}(\cdot-y))(x)|^qdx \leq 2^{(-j+1)\varepsilon} \nonumber \\
    \Rightarrow  & \int_{|x-y| \geq  2^{-j+1}}\bigg(\frac{|x-y|}{2^{-j+1}}\bigg)^{n(\frac{q}{p}-1)+\varepsilon}|T(\psi_j(\cdot-y))(x)|^{q}dx \leq  C_{n,p,q} 2^{(j-1)n(q-1)}.
\end{align}
With the same argument, one can also show
\begin{align} \label{Mole4}
   &\int_{|x-y|\geq  2^{-j+1}}\bigg(\frac{|x-y|}{2^{-j+1}}\bigg)^{n(\frac{q}{p}-1)+\varepsilon}\Big|T[(-1)^{|\alpha|}\partial^{\alpha}\psi_j(\cdot-c)](x)\Big|^{q}dx\leq C_{n,p,q}2^{(j-1)[q|\beta|+(q-1)n]},
\end{align}
and whenever $|y-c|\leq r<\frac{2^{-j+1}}{3}$ 
\begin{align} \label{Mole5}
    \int_{|x-c|\leq  2^{-j+2}}|T(A_j)(x)|^{q}dx\leq C_{n,p,q}  2^{(j-2)[q(N_p+1)+(q-1)n]}|y-c|^{q(N_p+1)};
\end{align}
\begin{align} \label{Mole6}
    \int_{|x-c|\geq 2^{-j+2}}\bigg(\frac{|x-c|}{2^{-j+2}}\bigg)^{n(\frac{q}{p}-1)+\varepsilon}|T(A_j)(x)|^{q}\leq C_{n,p,q}2^{(j-2)[q(N_p+1)+(q-1)n]}|y-c|^{q(N_p+1)},
\end{align}
where $A_j$ is defined in Lemma \ref{atomlemma}.
Due to the assumption and the constants explicitly stated in Lemma \ref{atomlemma}, the implicit constants in the upper bound of all the above estimates are independent of $j$ and $y$. 

We first consider the case $r\geq 1$ and recall that $P_{K_j,B}(x,y)=0$. Since $x \in B(c,4r)^c$ and $y\in B(c,r)$ implies $|x-c|\leq |x-y|+|y-c|\leq \frac{4}{3}|x-y|$, from \eqref{Mole2} we get
\begin{align*}
&\bigg\|\bigg(\frac{|x-c|}{r}\bigg)^{n(\frac{1}{p}-\frac{1}{q})+\frac{\varepsilon}{q}}\|K_j(x,y)\|_{L^{q}(B,dy/|B|)}\bigg\|_{L^q([B(c,4r)]^c,dx)}^q \\
&= \fint_{|y-c|\leq r}\int_{|x-c|\geq 4r}\bigg(\frac{|x-c|}{r}\bigg)^{n(\frac{q}{p}-1)+\varepsilon} |K_j(x,y)|^q dxdy \\
&\leq  \fint_{|y-c|\leq r}\int_{|x-c|\geq 4r}\bigg(\frac{|x-y|+|y-c|}{r}\bigg)^{n(\frac{q}{p}-1)+\varepsilon} |K_j(x,y)|^q dxdy \\
&\leq C_{n,p,q, \varepsilon}' 2^{(-j+1)[n(\frac{q}{p}-1)+\varepsilon]} r^{-n(\frac{q}{p}-1)-\varepsilon} \fint_{|y-c|\leq r}\int_{|x-y|\geq 3r}\bigg(\frac{|x-y|}{2^{-j+1}}\bigg)^{n(\frac{q}{p}-1)+\varepsilon} |K_j(x,y)|^q dxdy \\
&\leq C_{n,p,q, \varepsilon}'' 2^{(-j+1)[n(\frac{q}{p}-1)+\varepsilon]} r^{-n(\frac{q}{p}-1)-\varepsilon}2^{(j-1)n(q-1)} \leq  C_{n,p,q, \varepsilon}'' 2^{(-j+1)(nq\left(1/p-1\right))}r^{n(1-q)},
\end{align*}
where in the last line we have used that $r^{nq(1-1/p)-\varepsilon}\leq 1$. After taking both sides the $q$-th root, we have summable coefficients over $j$ and the desired factor $r^{n(1/q-1)}$.

Consider now the case $r<1$. We show \eqref{maingoal} for 
$$ 
P_{K_j,B(c,r)}(x,y)=\sum_{0\leq |\alpha|\leq N_p} \frac{1}{\alpha!} T[(-1)^{|\alpha|}\partial^{\alpha}\psi_j(\cdot-c)](x) (y-c)^{\alpha}
$$ 
We first split the sum over $j$ into
\begin{align*}
&\sum_{j=0}^{\infty}\bigg\|\bigg(\frac{|x-c|}{r}\bigg)^{n(\frac{1}{p}-\frac{1}{q})+\frac{\varepsilon}{q}}\|K_j(x,y)-P_{K_j,B}(x,y)\|_{L^{q}(B,dy/|B|)}\bigg\|_{L^q([B(c,4r)]^c,dx)} \\
&= \bigg(\sum_{j: 2^{-j+1}\leq 3r}+\sum_{j:2^{-j+1}> 3r}\bigg) \\
& \quad\quad \times \bigg\|\bigg(\frac{|x-c|}{r}\bigg)^{n(\frac{1}{p}-\frac{1}{q})+\frac{\varepsilon}{q}}\|K_j(x,y)-P_{K_j,B}(x,y)\|_{L^{q}(B,dy/|B|)}\bigg\|_{L^q([B(c,4r)]^c,dx)}.
\end{align*}

When $j$ satisfies $2^{-j+1}\leq 3r$, by triangle inequality, it suffices to provide the desired estimate for $K_j(x,y)$ and $T[(-1)^{|\alpha|}\partial^{\alpha}\psi_j(\cdot-c)](x)$ separately. To deal with the estimate of $K_j(x,y)$, we use again that $|x-c|\leq |x-y|+|y-c|\leq \frac{4}{3}|x-y|$ and hence
\begin{align*}
    &\int_{|x-c|\geq 4r} \fint_{|y-c|\leq r} \bigg(\frac{|x-c|}{r}\bigg)^{n(\frac{q}{p}-1)+\varepsilon}|K_j(x,y)|^q dydx \\
    &\leq C_{n,p,q,\varepsilon}\fint_{|y-c|\leq r} \int_{|x-y|\geq 3r} \bigg(\frac{|x-y|}{r}\bigg)^{n(\frac{q}{p}-1)+\varepsilon}|K_j(x,y)|^q dxdy \\
    &\leq C_{n,p,q,\varepsilon}\bigg(\frac{2^{-j+1}}{r}\bigg)^{n(\frac{q}{p}-1)+\varepsilon} \fint_{|y-c|\leq r} \int_{|x-y|\geq 2^{-j+1}} \bigg(\frac{|x-y|}{r}\bigg)^{n(\frac{q}{p}-1)+\varepsilon}|K_j(x,y)|^q dxdy \\
    &\leq C_{n,p,q,\varepsilon}' \bigg(\frac{2^{-j+1}}{r}\bigg)^{n(\frac{q}{p}-1)+\varepsilon}  2^{(j-1)n(q-1)} \leq C_{n,p,q,\varepsilon}' (2^{-j+1})^{q\gamma_p+\varepsilon} r^{-n(\frac{q}{p}-1)-\varepsilon}.
\end{align*}
After taking $q$-th root and summing over $j$s, we have the upper bound $C_{n,p,q,\varepsilon}''r^{n(q^{-1}-1)}$.

To estimate the term $T[(-1)^{|\alpha|}\partial^{\alpha}\psi_j(\cdot-c)](x)(y-c)^{\alpha}$, we use \eqref{Mole4} and we obtain
\begin{align*}
    &\int_{|x-c|\geq 4r} \fint_{|y-c|\leq r} \bigg(\frac{|x-c|}{r}\bigg)^{n(\frac{q}{p}-1)+\varepsilon}\Big|T[(-1)^{|\alpha|}\partial^{\alpha}\psi_j(\cdot-c)](x)\Big|^q |(y-c)^{\alpha}|^q dydx \\
    &\leq C_{n,p,q,\varepsilon}\fint_{|y-c|\leq r} \int_{|x-y|\geq 3r} \bigg(\frac{|x-y|}{r}\bigg)^{n(\frac{q}{p}-1)+\varepsilon}\Big|T[(-1)^{|\alpha|}\partial^{\alpha}\psi_j(\cdot-c)](x)\Big|^q |(y-c)^{\alpha}|^q dxdy \\
    &\leq C_{n,p,q,\varepsilon}\bigg(\frac{2^{-j+1}}{r}\bigg)^{n(\frac{q}{p}-1)+\varepsilon} \\
    & \quad\quad \times \fint_{|y-c|\leq r} \int_{|x-y|\geq 2^{-j+1}} \bigg(\frac{|x-y|}{r}\bigg)^{n(\frac{q}{p}-1)+\varepsilon}\Big|T[(-1)^{|\alpha|}\partial^{\alpha}\psi_j(\cdot-c)](x)\Big|^q r^{q|\alpha|} dxdy \\
    &\leq C_{n,p,q,\varepsilon}' \bigg(\frac{2^{-j+1}}{r}\bigg)^{n(\frac{q}{p}-1)+\varepsilon}  2^{(j-1)[n(q-1)+q|\alpha|]} r^{q|\alpha|} \\
    &\leq C_{n,p,q,\varepsilon}' (2^{-j+1})^{q\gamma_p+\varepsilon} r^{-n(\frac{q}{p}-1)-\varepsilon}.
\end{align*}
After taking $q$-th root and summing over those $j$s, we have the upper bound $C_{n,p,q,\varepsilon}''r^{n(q^{-1}-1)}$ thanks to the fact that $\varepsilon >0$ and $\gamma_p\geq N_p$.

Lastly, we need to handle the sum where $j$ satisfies $2^{-j+1}>3r$. Note that for a fixed $r<1$, there are only finitely many $j$ in this case. We rewrite $K_j(x,y)-P_{K_j,B}(x,y)$ as
$$K_j(x,y)-P_{K_j,B}(x,y) = T(A_j)(x),$$
where $A_j$ is given by \eqref{Ax}. Then, from the estimates \eqref{Mole5} and \eqref{Mole6} we get
\begin{align*}
    &\int_{|x-c|\geq 4r} \fint_{|y-c|\leq r} \bigg(\frac{|x-c|}{r}\bigg)^{n\left(\frac{q}{p}-1\right)+\varepsilon}|T(A_j)(x)|^q dydx \\
    &\leq C_{n,p,q,\varepsilon}\fint_{|y-c|\leq r} \int_{|x-y|\geq 3r} \bigg(\frac{|x-y|}{r}\bigg)^{n\left(\frac{q}{p}-1\right)+\varepsilon}|T(A_j)(x)|^q dxdy \\
    &\leq C_{n,p,q,\varepsilon}\bigg(\fint_{|y-c|\leq r}\int_{3r\leq |x-y|\leq 2^{-j+2}} \\
    &\quad\quad\quad\quad\quad\quad +\fint_{|y-c|\leq r}\int_{|x-y|\geq 2^{-j+2}}\bigg)\bigg(\frac{|x-y|}{r}\bigg)^{n\left(\frac{q}{p}-1\right)+\varepsilon} |T(A_j)(x)|^qdydx\\
    &\leq C_{n,p,q,\varepsilon}'  \bigg(\frac{2^{-j+2}}{r}\bigg)^{n\left(\frac{q}{p}-1\right)+\varepsilon} 2^{(j-2)[q(N_p+1)+(q-1)n]}r^{q(N_p+1)} \\ 
    &\leq C_{n,p,q,\varepsilon}'' (2^{-j})^{q\gamma_p+\varepsilon-q(N_p+1)}r^{q(N_p+1)-n(\frac{q}{p}-1)-\varepsilon} .
\end{align*}
After taking $q$-th root and summing over $j$s, we have the upper bound $C_{n,p,q,\varepsilon}''r^{n(q^{-1}-1)}$ using the fact that $\varepsilon<q(N_p+1)-q\gamma_p$.
\end{proof}

We now are ready to prove the main theorem.

\begin{proof}[Proof of Theorem \ref{mainthm}]\ \\
     We may assume that $\varepsilon<s(N_p+1-\Gg_p)$ since an $(h^p,s,n(\frac{s}{p}-1)+\varepsilon, C)$ molecule is also an $(h^p,s,n(\frac{s}{p}-1)+\varepsilon', C)$ molecule for any $\varepsilon'<\varepsilon$. Moreover, since $T$ maps $(h^p, q')$ exact atoms to $(h^p,s,n(sp^{-1}-1)+\varepsilon,C)$ approximate molecules, we can express 
     $$
     P_{K,B}(x,y)=\sum_{j=0}^{\infty}P_{K_j,B}(x,y)=\sum_{|\alpha|\leq N_p}\frac{1}{\alpha!}F_{\alpha,c,r}(x)(y-c)^{\alpha}
     $$ 
     in $\mathcal{S}'(\R^n\times \R^n)$, where $F_{\alpha,c,r}=\sum_{j=0}^{\infty} \lambda_j M_j$, with $$\lambda_j = 2^{(j-1)|\alpha|}2^{(j-1)n}(2^n+1)\|\psi_0\|_{\mathcal{S}}|B(0,2^{-j+1})|^{\frac{1}{p}}$$ and $M_j(x) =\lambda_j^{-1} T[(-1)^{|\alpha|}\partial^{\alpha}\psi_j(\cdot-c)](x) $. 
    
    For this fixed ball $B=B(c,r)$, we will show that $F_{\alpha,c,r}(x)$
is a function when $x\notin B$ and is in $L^s(B^c)$. For simplicity, we write $\lambda=n(\frac{s}{p}-1)+\varepsilon$. By the definition of molecules, note that if $2^{-j+1}<r$,
$$\|M_j\|_{L^s(B^c)}\leq r^{-\lambda/s} \|M_j|\cdot-c|^{\lambda/s}\|_{L^s(B^c)}\leq C r^{-\frac{\lambda}{s}}2^{(-j+1)(\frac{\lambda}{s}-n(\frac{1}{p}-\frac{1}{s}))}$$
and if  $2^{-j+1}\geq r$ (only finitely many $j$s), then
\begin{align*}
    \|M_j\|_{L^s(B^c)}&\leq \|M_j\|_{L^s(B^c\cap B(c,2^{-j+1}))}+ \|M_j\|_{L^s([B(c,2^{-j+1})]^c)}\\
    &\leq C 2^{(j-1)n(\frac{1}{p}-\frac{1}{s})} + C  r^{-\frac{\lambda}{s}}2^{(-j+1)(\frac{\lambda}{s}-n(\frac{1}{p}-\frac{1}{s}))}.
\end{align*}
Then, since $\frac{\lambda}{s}=n(\frac{1}{p}-\frac{1}{s})+\frac{\varepsilon}{s}>n(\frac{1}{p}-\frac{1}{s})$, we have
\begin{align*}
    \sum_{j=0}^{\infty} \lambda_j \|M_j\|_{L^s(B^c)} &=\sum_{j: 2^{-j+1}\leq r} \lambda_j \|M_j\|_{L^s(B^c)}+\sum_{j: 2^{-j+1}\geq r} \lambda_j \|M_j\|_{L^s(B^c)} \\
    &\leq C_{n,p,s,\varepsilon}\sum_{j=0}^{\infty}  r^{-\frac{\lambda}{s}}2^{(-j+1)(\frac{\lambda}{s}-n(\frac{1}{p}-\frac{1}{s})+\gamma_p-|\alpha|)} + C_{n,p,s,\varepsilon}\sum_{j: 2^{-j+1}\geq r} 2^{(-j)n(\frac{1}{p}-\frac{1}{s})} \\
    &\leq C_{n,p,q,\varepsilon} (r^{-\frac{\lambda}{s}}+r^{-n\left(\frac{1}{p}-\frac{1}{s}\right)} )<\infty.
\end{align*}
Therefore, for fixed $B=B(c,r)$, $F_{\alpha,c,r}(x)$ is indeed a function when $|x-c|\geq r$, and therefore $P_{K,B}(x,y)$ is a function when $|x-c|\geq r$. Finally, the desired estimate follows from Proposition \ref{submainthm}.

If $N_p\neq \gamma_p$, note that 
    $$\sum_{j}\lambda_j^{p} \simeq_{n,p,s}\sum_{j} (2^{-j})^{p(n/p-n-|\alpha|)}\simeq_{n,p,s}\sum_{j} (2^{-j})^{p(\gamma_p-|\alpha|)}  <\infty$$
because $|\alpha|\leq N_p< \gamma_p$, and $M_j$ is an approximate molecule by Lemma \ref{atomlemma} and the hypothesis. Therefore, we can conclude that $F_{\alpha,c,r}(x)\in h^p(\R^n)$. 
\end{proof}

\begin{remark} \label{remark_end} \
\begin{enumerate}
    \item The assumption $q\leq s$ is to interchange the $L^q$ and $L^s$ norm in Theorem \ref{mainthm} and Proposition \ref{submainthm}. Indeed, the results in Theorem \ref{mainthm} and Proposition \ref{submainthm} are also true if we replace \eqref{Kernel1} by
\begin{align*} \label{Kernel2}
     \bigg\|\bigg\|\bigg(\frac{|x-c|}{r}\bigg)^{n(\frac{1}{p}-\frac{1}{s})+\frac{\varepsilon}{s}}\big(K(x,y)-P_{K,B}(x,y)\big)\bigg\|_{L^s([2B]^c,dx)}\bigg\|_{L^{q}(B,dy/|B|)}\leq C r^{n(s^{-1}-1)}
\end{align*} 
for $1\leq q,s\leq 2$. We need the order $\|\|\cdot\|_{L^q(B, dy/|B|)}\|_{L^s([2B]^c)}$ in the proof of Theorem \ref{mainthm0}, and in other places we consider $\|\|\cdot\|_{L^s([2B]^c)}\|_{L^q(B, dy/|B|)}$.
\item We can impose the weaker assumption on Proposition \ref{submainthm} that $T$ maps  $(h^p, q')$ atoms supported in $B(c,r)$ to $(h^p,s,n(\frac{s}{p}-1),\Ge)$ pre-molecules centered in $B(c,kr)$ for some $k\geq 1$ (independent of $c$ and $r$) by considering the cases $k2^{-j+1}\leq 3r$ and $k2^{-j+1}>3r$. The assumption of Theorem \ref{mainthm} can be weakened similarly.
\item The continuity of $T$ on $L^2(\R^n)$ is to avoid the cases in which the restriction of the extension of $T$ is not the same as $T$ originally defined (on $C^{\infty}_c(\R^n)$), see remark (i) after the proof of Theorem (1.16) in \cite{FHJW}; moreover, this assumption allows us to discuss $T$ acting on $h^p$ atoms.
\item In Theorem \ref{mainthm}, we are unable to show directly that $P_{K,B}(x,y)$ is the Taylor polynomial of $K(x,y)$ with respect to the $y$-variable because we may not always have $\partial^{\alpha}T(f)=T(\partial^{\alpha}f)$.
\item Theorem \ref{mainthm} also holds if we change the molecules in $h^p(\R^n)$ with the analogous one in $H^p(\R^n)$ (with $C=0$ in $(M_3)$). In this case, if $N_p\neq \gamma_p$, $F_{\alpha, c,r}\in H^p(\R^n)$ instead. In fact, in  Proposition \ref{submainthm}, nothing is used other than $(M_1)$ and $(M_2)$; and the molecules in $H^p$ and $h^p$ have the same size conditions $(M_1)$ and $(M_2)$. Therefore, Theorem \ref{mainthm} holds in this case.

\item We also remark that the results for Theorem \ref{mainthm0}, Theorem \ref{mainthm}, Corollary \ref{p=1case}, and the previous remark are true for Hardy spaces $H^p(\R^n)$ by removing the condition $P_{K,B}(x,y)=0$ if $r\geq 1$ and changing (ii) in Theorem \ref{mainthm0} to $T^*((x-c)^{\Ga})=0$ for all $|\alpha|\leq \lfloor n(\frac{1}{p}-1)\rfloor$. The proof of Theorem \ref{mainthm} is essentially the same by using homogeneous Littlewood-Paley type decomposition instead. In this case, there are infinitely many $j$ that satisfies $2^{-j+1}> 3r$ but $\sum_{2^{-j+1}>3r}(2^{j})^{q(N_p+1)-q\gamma_p-\varepsilon}$ is convergent. 
\end{enumerate}
\end{remark}

\begin{proof}[Proof of Corollary \ref{p=1case}]\ \\
First observe that $P_{K,B(c,r)}(x,y)=K(x,c)$ when $\frac{n}{n+1}<p\leq 1$ in the proof of Proposition \ref{submainthm}. If $r<1$ and $s>1$, 
\begin{align*}
    &\|\|K(x,y)-K(x,c)\|_{L^1(B,dy/|B|)}\|_{L^1([4B]^c,dx)} \\
    &\quad\quad\leq \|\|K(x,y)-K(x,c)\|_{L^q(B,dy/|B|)}\|_{L^1([4B]^c,dx)} \\
    &\quad\quad\leq  \|\|K(x,y)-K(x,c)\|_{L^q(B,dy/|B|)}|x-c|^{n(\frac{1}{p}-\frac{1}{s})+\frac{\varepsilon}{s}} |x-c|^{-[{n(\frac{1}{p}-\frac{1}{s})+\frac{\varepsilon}{s}}]}\|_{L^1([4B]^c,dx)} \\
    &\quad\quad\leq  \|\|K(x,y)-K(x,c)\|_{L^q(B,dy/|B|)}|x-c|^{n(\frac{1}{p}-\frac{1}{s})+\frac{\varepsilon}{s}} \|_{L^s([4B]^c,dx)} \\
    &\quad\quad\quad\quad\quad\quad \times \bigg(\int_{|x-c|\geq 4r}\frac{1}{|x-c|^{n(\frac{1}{p}-\frac{1}{s})s'+\frac{s'\varepsilon}{s}}} dx\bigg)^{\frac{1}{s'}} \\
    &\quad\quad\leq C r^{\frac{\varepsilon}{s}+n(\frac{1}{p}-1)} r^{-n(\frac{1}{p}-1)-\frac{\varepsilon}{s}}  =C'
\end{align*}
and the constant $C'$ is independent of $c$ and $r$. We have used the fact that $(\frac{1}{p}-\frac{1}{s})s'\geq 1$ if $p\leq 1< s$. The inequality holds in the case $s=1$ (and $p<1$) by considering $\sup_{|x-c|\geq 4r}$ instead of the integral.

If $r\geq 1$,  using \eqref{Mole2}, for all $1\leq s\leq 2$, 
\begin{align*}
    &\sum_{j=0}^{\infty}\int_{|x-c|\geq 4r} \bigg(\frac{|x-c|}{r}\bigg)^{n(\frac{s}{p}-1)+\varepsilon}|K_j(x,c)|^s dx    \\
    &\quad\quad\leq \sum_{j=0}^{\infty}\bigg(\frac{2^{-j+1}}{r}\bigg)^{n(\frac{s}{p}-1)+\varepsilon} \int_{|x-c|\geq 4r} \bigg(\frac{|x-c|}{2^{-j+1}}\bigg)^{n(\frac{s}{p}-1)+\varepsilon}|K_j(x,c)|^s dx \\
    &\quad\quad\leq C_{n,p,s}\sum_{j=0}^{\infty}\bigg(\frac{2^{-j+1}}{r}\bigg)^{n(\frac{s}{p}-1)+\varepsilon}2^{(j-1)n(s-1)} \leq C_{n,p,s}'
\end{align*}
Then a similar argument shows that  \begin{align*} 
    &\|\|K(x,y)-K(x,c)\|_{L^1(B,dy/|B|)}\|_{L^1([4B]^c,dx)} \\
    &\leq \|\|K(x,y)\|_{L^1(B,dy/|B|)}\|_{L^1([4B]^c,dx)}+\|\|K(x,c)\|_{L^1(B,dy/|B|)}\|_{L^1([4B]^c,dx)}\\
    &\leq C r^{\frac{\varepsilon}{s}+n(\frac{1}{p}-1)} r^{-n(\frac{1}{p}-1)-\frac{\varepsilon}{s}}+C  =C'.
\end{align*}

Therefore, by \cite{Suzuki0}*{Theorem 1}, we can conclude that $T$ can be extended to a linear operator that maps $L^1(\R^n)$ to $L^{1,\infty}(\R^n)$.
\end{proof}
\section*{Acknowledgment}
We would like to thank the referees for their valuable comments and for pointing out a mistake in Lemma \ref{atomlemma}. We also thank them for proposing a solution to fix such lemma.

\addcontentsline{toc}{section}{Bibliography}
\bibliographystyle{amsplain}
\bibliography{ref}

\end{document}